\documentclass[11pt,a4paper]{article}

\usepackage{amssymb,amsmath,amsthm}
\usepackage[all]{xypic}
\usepackage{mathptmx}

\newtheorem{thm}{Theorem}[section]
\newtheorem{prop}[thm]{Proposition}
\newtheorem{lem}[thm]{Lemma}
\newtheorem{cor}[thm]{Corollary}

\theoremstyle{definition}
\newtheorem{dfn}[thm]{Definition}
\newtheorem{rem}[thm]{Remark}
\newtheorem{prob}[thm]{Problem}
\renewcommand{\proofname}{{\rm\textbf{Proof.}}}

\def\dim{\mathop{\mathrm{dim}}\nolimits}
\def\Im{\mathop{\mathrm{Im}}\nolimits}
\def\Ker{\mathop{\mathrm{Ker}}\nolimits}
\def\Hom{\mathop{\mathrm{Hom}}\nolimits}
\def\Ext{\mathop{\mathrm{Ext}}\nolimits}
\def\id{\mathop{\mathrm{id}}\nolimits}
\def\ev{\mathop{\mathrm{ev}}\nolimits}
\def\Specm{\mathop{\mathrm{Specm}}\nolimits}
\def\supp{\mathop{\mathrm{supp}}\nolimits}
\def\mod{\mathop{\mathrm{mod}}\nolimits}
\def\Der{\mathop{\mathrm{Der}}\nolimits}

\def\z{\mathbb{Z}}
\def\c{\mathbb{C}}

\def\g{\mathfrak{g}}
\def\h{\mathfrak{h}}
\def\n{\mathfrak{n}}
\def\m{\mathfrak{m}}
\def\P{\mathcal{P}}
\def\F{\mathcal{F}}

\title{{\bfseries Extensions between finite-dimensional simple modules\\ over a generalized current Lie algebra}}

\author{Ryosuke Kodera}

\date{}

\makeatletter
\let\@old@@maketitle=\@maketitle
\def\@maketitle{%
\footnotetext{%
\hspace*{-1em}\hspace*{-\footnotesep}%
Graduate School of Mathematical Sciences, The University of Tokyo, 3-8-1 Komaba, Meguro-ku, Tokyo 153-8914, Japan.\\
E-mail address: kryosuke@ms.u-tokyo.ac.jp\\
2000 \textit{Mathematics Subject Classification.} 17B65, 17B10, 16E30.}
\@old@@maketitle
}
\makeatother

\begin{document}
\maketitle

\begin{abstract}
We calculate the first extension groups for finite-dimensional simple modules over an arbitrary generalized current Lie algebra, which includes the case of loop Lie algebras and their multivariable analogs.
\end{abstract}

\section{Introduction}
In this article we are concerned with finite-dimensional modules over a generalized current Lie algebra $A \otimes \g$, where $\g$ is a finite-dimensional semisimple Lie algebra defined over the complex number field $\c$ and $A$ is a nonzero finitely generated commutative $\c$-algebra.
This class of Lie algebras includes loop Lie algebras and their multivariable analogs.
Since the category of finite-dimensional $A \otimes \g$-modules is not semisimple in general, we need to study its homological properties.
The purpose of this article is to give an answer for the following problem which naturally arises during the study.

\begin{prob}
Calculate $\Ext^1(V, V')$ for any finite-dimensional simple $A \otimes \g$-modules $V,V'$.
\end{prob}

This work can be regarded as both a refinement and a generalization of one by Chari and Moura \cite{MR2078944}, which determines the blocks of the category of finite-dimensional modules over a loop Lie algebra.
One of the main tools used in \cite{MR2078944} is a family of the universal finite-dimensional highest weight modules called Weyl modules.
In \cite{MR2078944} some knowledge on composition factors of Weyl modules is established (See Corollary~\ref{cor:weyl_factor}) and they use it to determine the blocks.
The notion of Weyl modules is generalized by Feigin and Loktev \cite{MR2102326} for a general $A$.
They also prove the properties of Weyl modules mentioned above in a general situation.
Then techniques used in \cite{MR2078944} are applicable for a general case and in fact yield a stronger result than the block decomposition of the category.

We also refer a work by Chari and Greenstein \cite{MR2189867}.
They obtain a similar result on calculation of the first extension groups for the case of current Lie algebras by a different approach.
See Remark~\ref{rem:cg} and \ref{rem:cg2} for a more precise explanation.

Now we state the main result.
We denote by $V_\m(\lambda)$ the evaluation module associated with the finite-dimensional simple $\g$-module $V(\lambda)$ with highest weight $\lambda$ at a maximal ideal $\m$ of $A$.

\begin{thm}\label{thm:main_intro}
Let $V,V'$ be finite-dimensional simple $A \otimes \g$-modules.
\begin{enumerate}
\item 
If $\Ext^1(V, V') \neq 0$ then
\[V \simeq V_{\m_1}(\lambda_1) \otimes \cdots \otimes V_{\m_{r-1}}(\lambda_{r-1}) \otimes V_{\m_r}(\lambda_r)\]
and
\[V' \simeq V_{\m_1}(\lambda_1) \otimes \cdots \otimes V_{\m_{r-1}}(\lambda_{r-1}) \otimes V_{\m_r}(\lambda'_r)\]
for some $r \in \z_{\geq 0}, \m_1, \ldots, \m_r \in \Specm A, \lambda_1, \ldots, \lambda_r, \lambda'_r \in P^+$.

\item Suppose that
\[V = V_{\m_1}(\lambda_1) \otimes \cdots \otimes V_{\m_{r-1}}(\lambda_{r-1}) \otimes V_{\m_r}(\lambda_r)\]
and
\[V' = V_{\m_1}(\lambda_1) \otimes \cdots \otimes V_{\m_{r-1}}(\lambda_{r-1}) \otimes V_{\m_r}(\lambda'_r)\]
where $\lambda_r$ and $\lambda'_r$ are possibly equal to zero.

If $\lambda_r \neq \lambda'_r$ then
\begin{align*}
\Ext^1(V, V') &\simeq \Ext^1(V_{\m_r}(\lambda_r), V_{\m_r}(\lambda'_r))\\
&\simeq \Hom_{\mathfrak{g}}(\mathfrak{g} \otimes V(\lambda_r), V(\lambda'_r)) \otimes \Der(A, A/\m_r).
\end{align*}

If $\lambda_r = \lambda'_r$ then
\begin{align*}
\Ext^1(V, V') &\simeq \displaystyle{\bigoplus_{i=1}^r} \Ext^1(V_{\m_i}(\lambda_i), V_{\m_i}(\lambda_i))\\
&\simeq \bigoplus_{i=1}^r \bigl(\Hom_{\mathfrak{g}}(\mathfrak{g} \otimes V(\lambda_i), V(\lambda_i)) \otimes \Der(A, A/\m_i)\bigr).
\end{align*}
\end{enumerate}
\end{thm}

By the above result, it turns out that extensions between simple modules rely on the choice of a vector of the Zariski tangent space at each point of $\Specm A$.

The article is organized as follows.
Section~2 is devoted to recall some definitions and fundamental facts.
It contains the definition of generalized current Lie algebras, the classification of finite-dimensional simple modules and various properties of Weyl modules.
The main theorem is proved in Section~3.
In Section~4 we consider the block decomposition of the category of finite-dimensional modules over a generalized current Lie algebra.
This generalizes the result by Chari and Moura \cite{MR2078944}.

\subsection*{Acknowledgments}
The author is grateful to Noriyuki Abe for fruitful discussions.
His suggestion led the author to consider a more general situation than the case of loop Lie algebras.
He would like to thank Katsuyuki Naoi, Noriyuki Abe and Yoshihisa Saito who read the manuscript carefully and gave helpful comments.
Finally he would like to express his gratitude to Yoshihisa Saito for his variable advice and kind support.


\section{Finite-dimensional modules over a generalized current Lie algebra}

\subsection{Semisimple Lie algebras}
Let $\mathfrak{g}$ be a finite-dimensional semisimple Lie algebra over the complex number field $\c$.
We denote by $\h$ a fixed Cartan subalgebra and $\n$ the nilpotent radical of a fixed Borel subalgebra containing $\h$.
Let $I$ be the index set of simple roots.
We choose Chevalley generators $e_i,h_i,f_i$ ($i \in I$) of $\g$.

We denote by $P$ the weight lattice and $Q$ the root lattice.
The set of dominant weights $P^+$ is defined by $P^+ = \{\lambda \in P \mid \langle h_i, \lambda \rangle \geq 0\ \text{for any $i \in I$}\}$.
For $\lambda, \mu \in P$ we say that $\lambda \geq \mu$ if $\lambda - \mu$ is expressed as a sum of simple roots with all nonnegative coefficients.

Let $V(\lambda)$ be the finite-dimensional simple $\mathfrak{g}$-module with highest weight $\lambda \in P^+$.
The highest weight of the dual module $V(\lambda)^*$ of $V(\lambda)$ is denoted by $\lambda^*$.

\subsection{Generalized current Lie algebras}
Let $\mathfrak{a}$ be an arbitrary Lie algebra over $\mathbb{C}$.
For a given nonzero finitely generated commutative $\c$-algebra $A$, we define the Lie algebra structure on the tensor product $A \otimes \mathfrak{a}$ by
\[[a \otimes x, b \otimes y] = ab \otimes [x,y]\]
for $a,b \in A$ and $x, y \in \mathfrak{a}$.

We call the Lie algebra $A \otimes \g$ the \emph{generalized current Lie algebra}.
The most familiar examples in this class of Lie algebras are the \emph{loop Lie algebra} for $A = \c[t,t^{-1}]$, the ring of Laurent polynomials in one variable and the \emph{current Lie algebra} for $A =\c[t]$, the ring of polynomials in one variable. 

\subsection{Simple modules}
We recall the classification of finite-dimensional simple $A \otimes \g$-modules given by Chari, Fourier and Khandai \cite{cfk}.
For each maximal ideal $\m$ of $A$, we define the \emph{evaluation homomorphism} at $\m$
\[\ev_\m \colon A \otimes \g \to \g\]
by
\[\ev_\m(a \otimes x) = a_\m x \]
for $a \in A$ and $x \in \g$, where $a_\m$ denotes the image of $a$ by the natural projection $A \to A/\m \simeq \c$.
This $\ev_\m$ is a surjective Lie algebra homomorphism. 
For a $\g$-module $V$ and a maximal ideal $\m$ of $A$, we can define the $A \otimes \g$-module structure on $V$ through $\ev_\m$.
We call it the \emph{evaluation module} associated with $V$ at $\m$ and denote by $\ev_\m^*(V)$.
We denote by $V_\m(\lambda)$ the evaluation module $\ev_\m^*(V(\lambda))$. 
This module $V_\m(\lambda)$ is simple.
Note that $V_\m(0) \simeq V_{\m'}(0)$ for any maximal ideals $\m, \m'$.
The following proposition is proved in \cite{cfk}.

\begin{prop}
\begin{enumerate}
\item The module $\bigotimes_{i=1}^r V_{\m_i}(\lambda_i)$ is simple if and only if $\m_1, \ldots, \m_r$ are all distinct.

\item Suppose that $\bigotimes_{i=1}^r V_{\m_i}(\lambda_i)$ and $\bigotimes_{i=1}^s V_{\m'_i}(\lambda'_i)$ are simple and $\lambda_1, \ldots, \lambda_r, \lambda'_1, \ldots, \lambda'_s$ are all nonzero.
Then $\bigotimes_{i=1}^r V_{\m_i}(\lambda_i)$ and $\bigotimes_{i=1}^s V_{\m'_i}(\lambda'_i)$ are isomorphic if and only if $r=s$ and the tuples $((\m_i, \lambda_i))_{1 \leq i \leq r}$ and $((\m'_i, \lambda'_i))_{1 \leq i \leq r}$ are same up to permutation.

\item Any finite-dimensional simple $A \otimes \g$-module is of the form $\bigotimes_{i=1}^r V_{\m_i}(\lambda_i)$.
\end{enumerate}
\end{prop}

Let $\mathcal{P}$ be the set of all functions from $\Specm A$ to $P^+$ with finite supports, where $\Specm A$ denotes the set of all maximal ideals of $A$.
The above proposition implies the classification of finite-dimensional simple $A \otimes \g$-modules.

\begin{thm}
The assignment
\[\pi \mapsto \bigotimes_{\m \in \supp \pi}V_\m(\pi(\m))\]
gives a one-to-one correspondence between $\mathcal{P}$ and the set of isomorphism classes of finite-dimensional simple $A \otimes \g$-modules. 
\end{thm}

We denote by $\mathcal{V}(\pi)$ the finite-dimensional simple $A \otimes \g$-module which corresponds to $\pi \in \mathcal{P}$.
For $\pi \in \mathcal{P}$ we define $\pi^* \in \mathcal{P}$ by $\pi^*(\m) = \pi(\m)^*$ for $\m \in \Specm A$.
The dual module $\mathcal{V}(\pi)^*$ of $\mathcal{V}(\pi)$ is isomorphic to $\mathcal{V}(\pi^*)$.

\subsection{Weyl modules}

\begin{dfn}
Let $V$ be an $A \otimes \g$-module.
A nonzero element $v \in V$ is called a \emph{highest weight vector} if $v$ is annihilated by $A \otimes \n$ and is a common eigenvector of $A \otimes \h$.
A module is called a \emph{highest weight module} if it is generated by a highest weight vector.
For a highest weight module $V$ generated by a highest weight vector $v$, there exists $\Lambda \in (A \otimes \h)^*$ such that
\[xv = \langle x, \Lambda \rangle v\]
for every $x \in A \otimes \h$.
This $\Lambda$ is called the \emph{highest weight} of $V$. 
\end{dfn}

\begin{rem}
The above definition of highest weight modules is consistent with the usual one for the case $A = \c$.
They are called $l$-highest weight modules for the case $A = \c[t,t^{-1}]$ in the literature.
\end{rem}

Any finite-dimensional simple $A \otimes \g$-module is a highest weight module.
Recall that such a module is of the form $\mathcal{V}(\pi)$ for some $\pi \in \mathcal{P}$.
We use the same symbol $\pi$ for the highest weight of $\mathcal{V}(\pi)$.
In other words we regard $\mathcal{P}$ as a subset of $(A \otimes \h)^*$ via the classification of simple modules.
To be explicit $\pi$ is determined by
\[\langle a \otimes h, \pi \rangle = \sum_{\m \in \supp \pi}a_\m \langle h, \pi(\m) \rangle\]
for $a \in A$ and $h \in \h$.
We identify $1 \otimes \h$ with $\h$.
Then the restriction $\pi$ to $1 \otimes \h$ is identified with the element $\sum_{\m \in \supp \pi}\pi(\m) \in P^+$.
We denote by $\pi |_{\h}$ this element.
 
\begin{dfn}
Let $\pi$ be an element of $\mathcal{P}$.
The \emph{Weyl module} $\mathcal{W}(\pi)$ is the $A \otimes \g$-module generated by a nonzero element $v_\pi$ with the following defining relations:
\[(A \otimes \n)v_\pi = 0,\]
\[xv_\pi = \langle x, \pi \rangle v_\pi\]
for $x \in A \otimes \h,$
\[(1 \otimes f_i)^{\langle h_i, \pi |_{\h}\rangle +1} v_\pi = 0\]
for $i \in I$.
\end{dfn}

By the definition of the Weyl module $\mathcal{W}(\pi)$, any finite-dimensional highest weight module with highest weight $\pi$ is a quotient of $\mathcal{W}(\pi)$. 
In particular the simple module $\mathcal{V}(\pi)$ is the unique simple quotient of $\mathcal{W}(\pi)$.
We denote by $W_{\m}(\lambda)$ the Weyl module which has the simple quotient $V_{\m}(\lambda)$.
The notion of Weyl modules for the case $A = \c[t, t^{-1}]$ is introduced by Chari and Pressley~\cite{MR1850556} and the following fundamental results are proved.
Later they are generalized by Feigin and Loktev~\cite{MR2102326} for a general $A$.

\begin{thm}\label{thm:weyl}
\begin{enumerate}
\item Any Weyl module is finite-dimensional.
\item We have
\[\mathcal{W}(\pi) \simeq \bigotimes_{\m \in \supp \pi}W_{\m}(\pi(\m))\]
for any $\pi \in \mathcal{P}$.
\end{enumerate}
\end{thm}

The following proposition is proved for the case $A=\c[t, t^{-1}]$ in \cite{MR2078944} and for a general case in \cite{MR2102326}.

\begin{prop}\label{prop:suff}
For a sufficiently large $k$, we have
\[(\mathfrak{m}^k \otimes \h)W_{\mathfrak{m}}(\lambda) = 0.\]
\end{prop}

\begin{cor}\label{cor:weyl_factor}
\begin{enumerate}
\item Any composition factor of $W_\m(\lambda)$ is of the form $V_\m(\mu)$ for some $\mu \in P^+$.
\item Any composition factor of $\mathcal{W}(\pi)$ is of the form $\mathcal{V}(\pi')$ such that $\supp\pi' \subseteq \supp \pi$. 
\end{enumerate}
\end{cor}
\begin{proof}
The assertion of (i) is deduced from the following fact: for distinct maximal ideals $\m$ and $\m'$, we have $\m^k \nsubseteq \m'$ for any $k$.
 
The assertion of (ii) is an immediate consequence of (i) and Theorem~\ref{thm:weyl} (ii).
\end{proof}

\begin{rem}
The assertions of this corollary for the case $A = \c[t, t^{-1}]$ is proved in \cite{MR2078944} and used for the proof of vanishing of the extension groups for certain modules.
We will also use it to prove vanishing of extension groups (Lemma~\ref{lem:distinct}) under an assumption slightly different from one in \cite{MR2078944}.
\end{rem}


\section{Extensions between simple modules}\label{main}
We denote by $\Ext^1$ the first Yoneda extension functor for finite-dimensional $A \otimes \g$-modules.
The purpose of this section is to calculate $\Ext^1(V,V')$ for any finite-dimensional simple $A \otimes \g$-modules $V,V'$.
 
\subsection{Extensions between evaluation modules} 
A derivation of $A$ into an $A$-module $M$ is a $\c$-linear map $D \colon A \to M$ satisfying
\[D(ab) = aD(b) + bD(a)\]
for $a,b \in A$.
We denote by $\Der(A, M)$ the $\c$-vector space of all derivations of $A$ into $M$. 
The following proposition is a special case of the main theorem. 

\begin{prop}\label{prop:single}
We have an isomorphism
\[\Ext^1(V_\m(\lambda), V_\m(\mu)) \simeq \Hom_{\mathfrak{g}}(\mathfrak{g} \otimes V(\lambda), V(\mu)) \otimes \Der(A, A/\m)
.\]
\end{prop}
\begin{proof}
We prove the assertion by the following steps.
\begin{description}

\item[(Step 1)] Define a map
\begin{align*}
& \Ext^1(V_\m(\lambda), V_\m(\mu))\\
& \to \{\varphi \colon A \to \Hom_{\g}(\g \otimes V(\lambda), V(\mu)) \mid \text{$\varphi$ is $\c$-linear, } \varphi(ab) = a_\m\varphi(b) + b_\m\varphi(a)\}.
\end{align*}

\item[(Step 2)] Show that the map is bijective by constructing the inverse map.

\item[(Step 3)] Show that the map is $\c$-linear.

\item[(Conclusion)] Assume that the above steps are proved.
It is obvious that
\[\{\varphi \colon A \to \Hom_{\g}(\g \otimes V(\lambda), V(\mu)) \mid \text{$\varphi$ is $\c$-linear, } \varphi(ab) = a_\m\varphi(b) + b_\m\varphi(a)\}\]
is canonically isomorphic to
\[\Hom_{\mathfrak{g}}(\mathfrak{g} \otimes V(\lambda), V(\mu)) \otimes \Der(A, A/\m).\]
Then we obtain an isomorphism as required.
\end{description}
We start to prove Step 1-3.\\

\noindent {\bf (Step 1)} Suppose that an exact sequence
\[\xymatrix{0 \ar[r] & V_\m(\mu) \ar[r]^{i} & E \ar[r]^{p} & V_\m(\lambda) \ar[r] & 0}\]
is given.
Take a splitting $j \colon V_\m(\lambda) \to E$ as $\g$-modules.
We identify $V_\m(\lambda)$ with $V(\lambda)$ and $V_\m(\mu)$ with $V(\mu)$ as $\g$-modules by restriction.
Then we define the $\c$-linear map $\varphi_a \colon \g \otimes V(\lambda) \to V(\mu)$ for each $a \in A$ via the action of $A \otimes \g$ on $E$ by 
\[(a \otimes x) j(u) = a_\m j(xu) + i(\varphi_a(x \otimes u))\]
for $x \in \g$ and $u \in V(\lambda)$.
Note that $a \mapsto \varphi_a$ defines a $\c$-linear map and $\varphi_1 = 0$.
We claim the followings:
\begin{description}
\item[{\rm (*-1)}] $\varphi_a$ does not depend on the choice of a splitting,
\item[{\rm (*-2)}] $\varphi_a$ depends only on the extension class of a given exact sequence.
\end{description}
To show (*-1), take another splitting $j'$ and let $\varphi_a'$ be the corresponding $\c$-linear map.
Then we have
\[i ((\varphi_a - \varphi_a')(x \otimes u)) = (a \otimes x)(j - j')(u) - a_\m x (j - j')(u).\]
The right-hand side is equal to zero since $(j - j')(u) \in \Ker p = \Im i$.
This shows (*-1). 
We show (*-2). 
Take two exact sequences which are equivalent: 
\[\xymatrix{0 \ar[r] & V_\m(\mu) \ar[r]^{i}\ar@{=}[d] & E \ar[r]^{p}\ar[d]^{\xi} & V_\m(\lambda) \ar[r]\ar@{=}[d] & 0 \\
0 \ar[r] & V_\m(\mu) \ar[r]^{i'} & E' \ar[r]^{p'} & V_\m(\lambda) \ar[r] & 0.}\]
Let $\varphi_a, \varphi_a'$ be the corresponding maps.
Splittings $j$ of $p$ and $j'$ of $p'$ can be taken so that $j' = \xi j$.
We have 
\[(a \otimes x)j(u) = a_\m j(xu) + i (\varphi_a(x \otimes u)),\]
\[(a \otimes x)j'(u) = a_\m j'(xu) + i' (\varphi_a'(x \otimes u))\]
by the definition of $\varphi_a, \varphi_a'$.
We see that $\varphi_a = \varphi_a'$ by applying $\xi$ to the both sides of the first equation and comparing it with the second one.
The claim is proved.

We show that $\varphi_a$ is a $\g$-module homomorphism and the equation
\[\varphi_{ab} = a_\m\varphi_{b} + b_\m\varphi_{a}\]
holds.
We have
\[(a \otimes x)(b \otimes y) j(u) = a_\m b_\m j(xyu) + a_\m i(x \varphi_b(y \otimes u)) + b_\m i(\varphi_a(x \otimes yu))\]
and hence
\begin{align*}
&(a \otimes x)(b \otimes y) j(u) - (b \otimes y)(a \otimes x) j(u)\\
& = a_\m b_\m j([x,y]u) + a_\m i(x \varphi_b(y \otimes u) - \varphi_b(y \otimes xu)) + b_\m i(\varphi_a(x \otimes yu) - y \varphi_a (x \otimes u)).
\end{align*}
Compare the above with
\[(ab \otimes [x,y]) j(u) = a_\m b_\m j([x,y]u) + i(\varphi_{ab}([x,y] \otimes u))\]
and we obtain
\[\varphi_{ab}([x,y] \otimes u) = a_\m (x \varphi_b(y \otimes u) - \varphi_b(y \otimes xu)) + b_\m (\varphi_a(x \otimes yu) - y \varphi_a(x \otimes u)).\]
Consider the case $b=1$.
Then we obtain the equation
\[\varphi_{a}([x,y] \otimes u) = \varphi_a (x \otimes yu) - y \varphi_a (x \otimes u).\]
This proves that $\varphi_a$ is a $\g$-module homomorphism.
Moreover we have
\[\varphi_{ab}([x, y] \otimes u) = a_\m\varphi_{b} ([x, y] \otimes u) + b_\m\varphi_{a} ([x, y] \otimes u)\]
and this implies that
\[\varphi_{ab} = a_\m\varphi_{b} + b_\m\varphi_{a}\]
since $[\g, \g]=\g$.

As a result we obtain a $\c$-linear map $\varphi \colon A \to \Hom_\g (\g \otimes V(\lambda), V(\mu))$ satisfying 
\[\varphi(ab) = a_\m\varphi(b) + b_\m\varphi(a).\]
This means that a map
\begin{align*}
& \Ext^1(V_\m(\lambda), V_\m(\mu))\\
& \to \{\varphi \colon A \to \Hom_{\g}(\g \otimes V(\lambda), V(\mu)) \mid \text{$\varphi$ is $\c$-linear, } \varphi(ab) = a_\m\varphi(b) + b_\m\varphi(a)\}
\end{align*}
is defined.\\

\noindent {\bf (Step 2)} Conversely if a $\c$-linear map $\varphi \colon A \to \Hom_\g(\g \otimes V(\lambda), V(\mu))$ satisfying
\[\varphi(ab) = a_\m\varphi(b) + b_\m\varphi(a)\]
is given then we can define the $A \otimes \g$-module structure on $E = V(\lambda) \oplus V(\mu)$ by
\[(a \otimes x)(u,v) = (a_\m xu, a_\m xv + \varphi(a)(x \otimes u))\]
for $u \in V(\lambda), v \in V(\mu)$.
It is obvious that this gives the inverse of the map defined in Step 1.\\

\noindent {\bf (Step 3)} We show that the bijective map is $\c$-linear.
First we show that it is additive.
Let 
\[\xymatrix{0 \ar[r] & V_\m(\mu) \ar[r]^{i_1} & E_1 \ar[r]^{p_1} & V_\m(\lambda) \ar[r] & 0},\]
\[\xymatrix{0 \ar[r] & V_\m(\mu) \ar[r]^{i_2} & E_2 \ar[r]^{p_2} & V_\m(\lambda) \ar[r] & 0}\]
be exact sequences and $\varphi^1,\varphi^2$ be the corresponding elements.
The Bear sum of the classes of the above extensions is represented by
\[\xymatrix{0 \ar[r] & V_\m(\mu) \ar[r]^{i} & E \ar[r]^{p} & V_\m(\lambda) \ar[r] & 0}\]
where $E$ is the quotient of the fibered product of $p_1$ and $p_2$ by $\Im(v \mapsto (i_1(v), -i_2(v)))$.
Note that $i$ is given by $v \mapsto (i_1(v),0) = (0, i_2(v))$ in $E$ and $p$ by $(z_1, z_2) \mapsto p_1(z_1) = p_2(z_2)$.
A splitting $j$ of $p$ as $\g$-modules is given by $u \mapsto (j_1(u), j_2(u))$ if we take splittings $j_1$ of $p_1$ and $j_2$ of $p_2$.
Then the equation
\begin{align*}
(a \otimes x)j(u)
& = (a_\m j_1(xu) + i_1(\varphi^1_a (x \otimes u)), a_\m j_2(xu) + i_2(\varphi^2_a (x \otimes u))) \\
& = a_\m j(xu) + i((\varphi^1_a + \varphi^2_a)(x \otimes u))
\end{align*}
in $E$ holds for $a \in A$ and $x \in \g$.
This shows that the map under consideration is additive.
Next we consider the multiplication by scalar.
Take an exact sequence
\[\xymatrix{0 \ar[r] & V_\m(\mu) \ar[r]^{i} & E \ar[r]^{p} & V_\m(\lambda) \ar[r] & 0}\]
and let $\varphi$ be the corresponding element.
The action of $c \in \c$ on $\Ext^1(V_\m(\lambda), V_\m(\mu))$ is described by the diagram
\[\xymatrix{0 \ar[r] & V_\m(\mu) \ar[r]^{i'}\ar@{=}[d] & E' \ar[r]^{p'}\ar[d] & V_\m(\lambda) \ar[r]\ar[d]^{c\id} & 0 \\
0 \ar[r] & V_\m(\mu) \ar[r]^{i} & E \ar[r]^{p} & V_\m(\lambda) \ar[r] & 0}\]
where $E'$ is the fibered product of $p$ and $c\id_{V_\m(\lambda)}$.
Note that $i'$ is given by $v \mapsto (i(v),0)$ and $p'$ by the second projection.
A splitting $j'$ of $p'$ is given by $u \mapsto (cj(u),u)$ where $j$ is a splitting of $p$.
Then we obtain
\begin{align*}
(a \otimes x)j'(u)
& = (c(a_\m j(xu) + i(\varphi_a (x \otimes u))), a_\m xu) \\
& = a_\m j'(xu) + i'(c\varphi_a(x \otimes u)).
\end{align*}
The proof is complete.
\end{proof}

\begin{rem}
In \cite[Proposition~3.4]{MR2078944} Chari and Moura define the map from $\Hom_{\g}(\g \otimes V(\lambda), V(\mu))$ to $\Ext^1(V_\m(\lambda), V_\m(\mu))$ as in Step 2 of the proof for the case $A = \c[t,t^{-1}]$.
We follow their idea here.
In \cite{MR2078944} the space $\Der(A, A/\m)$ is one-dimensional and its contribution is not recognized explicitly.
\end{rem}

\subsection{A key lemma}
In this subsection we show a key lemma (Lemma~\ref{lem:triv}) to prove the main theorem.

The proof of the following lemma is a copy of an argument in \cite[Lemma~5.2]{MR2078944}.
While they prove vanishing of $\Ext^1$ for modules with different spectral characters (See Section~\ref{block} for the definition of spectral characters), we show a slightly different statement.

\begin{lem}\label{lem:distinct}
Let $\pi, \pi'$ be elements of $\mathcal{P}$ and suppose that $\supp \pi \cap \supp \pi' = \varnothing$.
If $\Ext^1(\mathcal{V}(\pi), \mathcal{V}(\pi')) \neq 0$ then $\pi$ or $\pi'$ is equal to zero.
\end{lem}
\begin{proof}
We may assume that either of $\pi$ or $\pi'$ is not equal to zero since $\Ext^1(\mathcal{V}(0), \mathcal{V}(0)) = 0$ by Proposition~\ref{prop:single}.
This assumption implies that $\pi \neq \pi'$.

Let
\[\xymatrix{0 \ar[r] & \mathcal{V}(\pi') \ar[r] & E \ar[r]^{p} & \mathcal{V}(\pi) \ar[r] & 0}\]
be a nonsplit exact sequence.
First we assume that $\pi' |_\h \not > \pi |_\h$.
Let $\mathcal{V}(\pi)_\pi$ be the one-dimensional subspace generated by a highest weight vector of $\mathcal{V}(\pi)$.
By the assumption $\pi' |_\h \not > \pi |_\h$, the subspace $p^{-1}(\mathcal{V}(\pi)_\pi)$ of $E$ is annihilated by $A \otimes \n$.
Since $p^{-1}(\mathcal{V}(\pi)_\pi)$ is stable by $A \otimes \h$, we can take a common eigenvector of $A \otimes \h$ in $p^{-1}(\mathcal{V}(\pi)_\pi)$ and denote it by $v$.
Then $v$ is a highest weight vector of $E$.
Consider the submodule of $E$ generated by $v$.
This submodule is not isomorphic to $\mathcal{V}(\pi')$ since their highest weights are different.
Then it follows that the submodule coincides with $E$ since the length of $E$ is two and the sequence does not split.
Hence $E$ is a highest weight module with highest weight $\pi$ and then a quotient of the Weyl module $\mathcal{W}(\pi)$. 
Therefore $\pi'$ must be equal to zero by Corollary \ref{cor:weyl_factor} and the assumption $\supp \pi \cap \supp \pi' = \varnothing$.
Next assume that $\pi' |_\h > \pi |_\h$.
In this case, take the dual of the exact sequence.
Then we obtain the exact sequence 
\[\xymatrix{0 \ar[r] & \mathcal{V}(\pi^*) \ar[r] & E^* \ar[r] & \mathcal{V}((\pi')^*) \ar[r] & 0}\]
and have $\pi^* |_\h \not > (\pi')^* |_\h$.
This implies that $\pi$ is equal to zero. 
\end{proof}

We recall an important fact (Corollary \ref{cor:rigid}) which will be used repeatedly in the sequel.
Let $M$ be a finite-dimensional $A \otimes \g$-module.
Then the exact functor $M \otimes -$ is defined.

\begin{prop}
The functor $M^* \otimes -$ is a right and left adjoint functor of $M \otimes -$.\end{prop}

This is a general fact which holds for the category of \emph{finite-dimensional} modules over a Hopf algebra with an involutive antipode defined over a field.
The proposition immediately implies the following.

\begin{cor}\label{cor:rigid}
We have the natural isomorphisms
\[\Ext^1(V, M \otimes V') \simeq \Ext^1(M^* \otimes V, V'),\]
\[\Ext^1(M \otimes V, V') \simeq \Ext^1(V, M^* \otimes V')\]
for $A \otimes \g$-modules $V,V',M$. 
\end{cor} 

\begin{rem}\label{rem:rigid}
We give explicit descriptions of the morphisms in Corollary~\ref{cor:rigid}.
The morphism
\[\Ext^1(M^* \otimes V, V') \to \Ext^1(V, M \otimes V')\]
is described as follows.
Let
\[\xymatrix{0 \ar[r] & V' \ar[r] & E \ar[r] & M^* \otimes V \ar[r] & 0}\]
be an exact sequence which represents an extension class in $\Ext^1(M^* \otimes V, V')$.
Then the corresponding element of $\Ext^1(V, M \otimes V')$ is represented by the first row of the diagram
\[\xymatrix{0 \ar[r] & M \otimes V' \ar[r]\ar@{=}[d] & E' \ar[r]\ar[d] & V \ar[r]\ar[d] & 0 \\
0 \ar[r] & M \otimes V' \ar[r] & M \otimes E \ar[r] & M \otimes M^* \otimes V \ar[r] & 0}\]
where $E'$ is the fibered product which makes the right square cartesian.
The other morphisms are obtained in similar ways.
\end{rem}

\begin{lem}\label{lem:triv}
Let $\pi$ be an element of $\mathcal{P}$.
We have $\Ext^1(\mathcal{V}(\pi), \mathcal{V}(0)) = 0$ and $\Ext^1(\mathcal{V}(0), \mathcal{V}(\pi)) = 0$ unless $\#\supp\pi =1$.
\end{lem}
\begin{proof}
Assume that $\#\supp\pi \geq 2$.
Then we can divide $\supp\pi = \{\m\} \sqcup \supp\pi'$ for some $\m$ and nonzero $\pi'$.
We have $\mathcal{V}(\pi) \simeq V_\m(\pi(\m)) \otimes \mathcal{V}(\pi')$.
Hence
\[\Ext^1(\mathcal{V}(\pi), \mathcal{V}(0)) \simeq \Ext^1(\mathcal{V}(\pi'), V_\m(\pi(\m)^*))\]
and the right-hand side is equal to zero by Lemma \ref{lem:distinct}.
The assertion $\Ext^1(\mathcal{V}(0), \mathcal{V}(\pi)) = 0$ is proved by taking the dual.

The assertion $\Ext^1(\mathcal{V}(0), \mathcal{V}(0)) = 0$ is a consequence of Proposition \ref{prop:single}.
\end{proof}

\begin{rem}\label{rem:cg}
In fact, by Proposition~\ref{prop:single}, it is easy to prove a stronger result than the statement of Lemma~\ref{lem:triv}.
We state it without a proof since it is not used in the sequel.
The followings are equivalent for a finite-dimensional simple $A \otimes \g$-module $V$:
\begin{itemize}
\item $\Ext^1(V, \mathcal{V}(0)) \neq 0,$
\item $\Ext^1(\mathcal{V}(0), V) \neq 0,$
\item $V \simeq V_\m(\theta)$ for some $\m \in \Specm A$ satisfying $\m / \m^2 \neq 0$, where $\theta$ denotes the highest root of $\g$.
\end{itemize}
This result for the case $A = \c[t]$ is proved in \cite{MR2189867} by a different approach.
They also prove that
\[\dim\Ext^1(V_\m (\theta), \mathcal{V}(0)) = \dim\Ext^1(\mathcal{V}(0), V_\m (\theta))=1\]
and deduce the following result:
\[\Ext^1(V, V') \simeq \bigoplus_{\m \in \Specm \c[t]}\Hom_{\c[t] \otimes \g}(V_\m(\theta), V^* \otimes V')\]
holds for any finite-dimensional simple $\c[t] \otimes \g$-modules $V, V'$.
\end{rem}

\subsection{Proof of the main theorem}

\begin{thm}\label{thm:main}
Let $\pi,\pi'$ be elements of $\P$.
\begin{enumerate}
\item 
If $\Ext^1(\mathcal{V}(\pi), \mathcal{V}(\pi')) \neq 0$ then
$\#\{\m \in \Specm A \mid \pi(\m) \neq \pi'(\m)\} \leq 1$.

\item 
If $\#\{\m \in \Specm A \mid \pi(\m) \neq \pi'(\m)\} = 1$ then
\begin{align*}
\Ext^1(\mathcal{V}(\pi), \mathcal{V}(\pi'))
& \simeq \Ext^1(V_{\m_0}(\pi(\m_0)), V_{\m_0}(\pi'(\m_0)))\\
& \simeq \Hom_{\g}(\g \otimes V(\pi(\m_0)), V(\pi'(\m_0))) \otimes \Der(A, A/\m_0)
\end{align*}
where $\m_0$ is the unique element of $\Specm A$ such that $\pi(\m_0) \neq \pi'(\m_0)$.

If $\pi = \pi'$ then
\begin{align*}
\Ext^1(\mathcal{V}(\pi), \mathcal{V}(\pi'))
& \simeq \bigoplus_{\m \in \supp\pi} \Ext^1(V_{\m}(\pi(\m)), V_{\m}(\pi(\m)))\\
& \simeq \bigoplus_{\m \in \supp\pi} \bigl(\Hom_{\mathfrak{g}}(\mathfrak{g} \otimes V(\pi(\m)), V(\pi(\m))) \otimes \Der(A, A/\m)\bigr).
\end{align*}
\end{enumerate}
\end{thm}

\begin{rem}\label{rem:cg2}
For the case $A = \c[t]$ it is proved in \cite{MR2189867} that
\[\Ext^1(V, V') \simeq \bigoplus_{\m \in \Specm \c[t]}\Hom_{\c[t] \otimes \g}(V_\m(\theta), V^* \otimes V')\]
holds for any finite-dimensional simple $\c[t] \otimes \g$-modules $V, V'$ as explained in Remark~\ref{rem:cg}.
This implies results similar to our main theorem after some calculation essentially same as the proof below.
\end{rem}

\renewcommand{\proofname}{{\rm\textbf{Proof of Theorem \ref{thm:main}.}}}
\begin{proof}
Recall that
\[\mathcal{V}(\pi) \simeq \bigotimes_{\m \in \supp\pi} V_{\m}(\pi(\m))\]
and
\[\mathcal{V}(\pi') \simeq \bigotimes_{\m \in \supp\pi'} V_{\m}(\pi'(\m)).\]

We prove (i).
Suppose that $\Ext^1(\mathcal{V}(\pi), \mathcal{V}(\pi')) \neq 0$.
Put
\begin{align*}
&S = \supp \pi \cap \supp \pi',\\
&T = \supp \pi \setminus S,\\
&T' = \supp \pi' \setminus S.
\end{align*}
Let
\[V(\pi(\m)) \otimes V(\pi'(\m))^* \simeq \bigoplus_{j_\m} V(\nu_{j_\m})\]
be a decomposition into a direct sum of simple $\mathfrak{g}$-modules.
Note that $\nu_{j_\m} = 0$ for some $j_\m$ if and only if $\pi(\m) = \pi'(\m)$.
We have
\begin{align*}
& \Ext^1(\mathcal{V}(\pi), \mathcal{V}(\pi'))\\ 
& \simeq \Ext^1(\bigotimes_{\m \in S} (V_{\m}(\pi(\m)) \otimes V_{\m}(\pi'(\m))^*) \otimes \bigotimes_{\m \in T} V_{\m}(\pi(\m)) \otimes \bigotimes_{\m \in T'} V_{\m}(\pi'(\m))^*, \mathcal{V}(0))\\
& \simeq \Ext^1(\bigotimes_{\m \in S} \ev_{\m}^*(V(\pi(\m)) \otimes V(\pi'(\m))^*) \otimes \bigotimes_{\m \in T}V_{\m}(\pi(\m)) \otimes \bigotimes_{\m \in T'} V_{\m}(\pi'(\m))^*, \mathcal{V}(0))\\
& \simeq \bigoplus_{(j_\m)_{\m \in S}} \Ext^1(\bigotimes_{\m \in S} V_{\m}(\nu_{j_\m}) \otimes \bigotimes_{\m \in T}V_{\m}(\pi(\m)) \otimes \bigotimes_{\m \in T'} V_{\m}(\pi'(\m))^*, \mathcal{V}(0)).
\end{align*}
There is a tuple $(j_\m)_{\m \in S}$ such that 
\[\Ext^1(\bigotimes_{\m \in S} V_{\m}(\nu_{j_\m}) \otimes \bigotimes_{\m \in T}V_{\m}(\pi(\m)) \otimes \bigotimes_{\m \in T'} V_{\m}(\pi'(\m))^*, \mathcal{V}(0)) \neq 0\]
by the assumption $\Ext^1(\mathcal{V}(\pi), \mathcal{V}(\pi')) \neq 0$.
By Lemma \ref{lem:triv}, the number of nontrivial factors of the tensor product is exactly one.
Hence one of the following three cases holds:
\begin{description}
\item[{\rm (*-1)}] $\pi(\m) = \pi'(\m)$ for all $\m \in S$ but at most one element and $T = T' = \varnothing,$

\item[{\rm (*-2)}] $\pi(\m) = \pi'(\m)$ for $\m \in S,$ $\# T= 1$ and $T' = \varnothing,$

\item[{\rm (*-3)}] $\pi(\m) = \pi'(\m)$ for $\m \in S,$ $T= \varnothing$ and $\# T' = 1$.
\end{description}
The case (*-1) implies that $\#\{\m \in \Specm A \mid \pi(\m) \neq \pi'(\m)\} \leq 1$ and the case (*-2) or (*-3) implies that $\#\{\m \in \Specm A \mid \pi(\m) \neq \pi'(\m)\} = 1$.
The proof of (i) is complete.

We prove (ii).
Suppose that $\#\{\m \in \Specm A \mid \pi(\m) \neq \pi'(\m)\} \leq 1$.
Put $U = \{\m \in \Specm A \mid \pi(\m) = \pi'(\m)\}$.
We can write as
\[\mathcal{V}(\pi) \simeq \bigotimes_{\m \in U}V_{\m}(\pi(\m)) \otimes V_{\m_0}(\pi(\m_0))\]
and
\[\mathcal{V}(\pi') \simeq \bigotimes_{\m \in U}V_{\m}(\pi(\m)) \otimes V_{\m_0}(\pi'(\m_0))\]
for some $\m_0$ where $\pi(\m_0)$ and $\pi'(\m_0)$ are possibly equal to zero.
Again let
\[V(\pi(\m)) \otimes V(\pi'(\m))^* \simeq \bigoplus_{j_\m} V(\nu_{j_\m})\]
be a decomposition into a direct sum of simple $\mathfrak{g}$-modules.
Then we have 
\begin{align*}
& \Ext^1(\mathcal{V}(\pi), \mathcal{V}(\pi')) \\
& \simeq \Ext^1(\bigotimes_{\m \in U} (V_{\m}(\pi(\m)) \otimes V_{\m}(\pi(\m))^*) \otimes (V_{\m_0}(\pi(\m_0)) \otimes V_{\m_0}(\pi'(\m_0))^*), \mathcal{V}(0))\\
& \simeq \bigoplus_{(j_\m)_{\m \in U \cup \{\m_0\}}} \Ext^1(\bigotimes_{\m \in U} V_{\m}(\nu_{j_\m}) \otimes V_{\m_0}(\nu_{j_{\m_0}}), \mathcal{V}(0)).
\end{align*}
By Lemma \ref{lem:triv}, the number of nontrivial factors of the tensor product is one in every nonzero summand.
If we suppose that $\pi(\m_0) \neq \pi'(\m_0)$ then $V_{\m_0}(\pi(\m_0)) \otimes V_{\m_0}(\pi'(\m_0))^*$ does not have a trivial direct summand.
Hence
\begin{align*}
\Ext^1(\mathcal{V}(\pi), \mathcal{V}(\pi'))
& \simeq \Ext^1(V_{\m_0}(\pi(\m_0)) \otimes V_{\m_0}(\pi'(\m_0))^*, \mathcal{V}(0))\\
& \simeq \Ext^1(V_{\m_0}(\pi(\m_0)), V_{\m_0}(\pi'(\m_0))).
\end{align*}
If $\pi = \pi'$, namely $U = \Specm A$ and $\pi(\m_0) = \pi'(\m_0) = 0$, then
\begin{align*}
\Ext^1(\mathcal{V}(\pi), \mathcal{V}(\pi))
& \simeq \Ext^1(\bigotimes_{\m \in \supp\pi} (V_{\m}(\pi(\m)) \otimes V_{\m}(\pi(\m))^*), \mathcal{V}(0)) \\
& \simeq \bigoplus_{\m \in \supp\pi} \Ext^1(V_{\m}(\pi(\m)) \otimes V_{\m}(\pi(\m))^*, \mathcal{V}(0))\\
& \simeq \bigoplus_{\m \in \supp \pi} \Ext^1(V_{\m}(\pi(\m)), V_{\m}(\pi(\m))).
\end{align*}
The proof of (ii) is complete together with Proposition \ref{prop:single}, which asserts the second isomorphisms.
\end{proof}
\renewcommand{\proofname}{{\rm\textbf{Proof.}}}

\begin{rem}\label{rem:main}
We give a natural interpretation of the isomorphisms
\[\Ext^1(\mathcal{V}(\pi), \mathcal{V}(\pi')) \simeq \Ext^1(V_{\m_0}(\pi(\m_0)), V_{\m_0}(\pi'(\m_0)))\]
and
\[\Ext^1(\mathcal{V}(\pi), \mathcal{V}(\pi)) \simeq \bigoplus_{\m \in \supp\pi} \Ext^1(V_{\m}(\pi(\m)), V_{\m}(\pi(\m)))\]
in Theorem \ref{thm:main}.
According to the proof, the above isomorphisms come from the composition of the morphisms
\[\Ext^1(V,V') \to \Ext^1(M^* \otimes M \otimes V, V') \simeq \Ext^1(M \otimes V, M \otimes V')\]
for appropriate modules $V,V',M$.
This morphism coincides with the natural morphism
\[\Ext^1(V, V') \to \Ext^1(M \otimes V, M \otimes V')\]
obtained by applying the exact functor $M \otimes -$.
This is proved as follows.
Let
\[\xymatrix{0 \ar[r] & V' \ar[r] & E \ar[r] & V \ar[r] & 0}\]
be an exact sequence which represents an extension class in $\Ext^1(V, V')$.
This element maps to the extension class represented by the first row of
\[\xymatrix{0 \ar[r] & V' \ar[r]\ar@{=}[d] & E' \ar[r]\ar[d] & M^* \otimes M \otimes V \ar[r]\ar[d] & 0 \\
0 \ar[r] & V' \ar[r] & E \ar[r] &  V \ar[r] & 0}\]
by $\Ext^1(V,V') \to \Ext^1(M^* \otimes M \otimes V, V')$ and then maps to the class represented by the first row of 
\[\xymatrix{0 \ar[r] & M \otimes V' \ar[r]\ar@{=}[d] & E'' \ar[r]\ar[d] & M \otimes V \ar[r]\ar[d] & 0 \\
0 \ar[r] & M \otimes V' \ar[r] & M \otimes E' \ar[r] & M \otimes M^* \otimes M \otimes V \ar[r] & 0}\]
by $\Ext^1(M^* \otimes M \otimes V, V') \to \Ext^1(M \otimes V, M \otimes V')$, as explained in Remark~\ref{rem:rigid}.
Consider the following diagram:
\[\xymatrix{0 \ar[r] & M \otimes V' \ar[r]\ar@{=}[d] & E'' \ar[r]\ar[d] & M \otimes V \ar[r]\ar[d] & 0 \\
0 \ar[r] & M \otimes V' \ar[r]\ar@{=}[d] & M \otimes E' \ar[r]\ar[d] & M \otimes M^* \otimes M \otimes V \ar[r]\ar[d] & 0\\
0 \ar[r] & M \otimes V' \ar[r] & M \otimes E \ar[r] &  M \otimes V \ar[r] & 0.}\]
Since the composition of the right vertical maps $M \otimes V \to M \otimes M^* \otimes M \otimes V \to M \otimes V$ is identity, the first and the third rows are equivalent.
\end{rem}


\section{The block decomposition}\label{block}

We deduce the block decomposition of the category of finite-dimensional $A \otimes \g$-modules from results of Section \ref{main} by the almost same argument in \cite{MR2078944}.
We give a proof for the sake of completeness.
In the sequel \emph{we assume that $A$ is connected}, namely it is not isomorphic to a direct product of two nonzero $\c$-algebras for simplicity.
Moreover we assume that $A \neq \c$ since the block decomposition is well-known for the case $A = \c$ as completely reducibility of finite-dimensional $\g$-modules.
Let $\Xi$ be the set of all functions from $\Specm A$ to $P/Q$ with finite supports.

\begin{dfn}
For each finite-dimensional simple $A \otimes \g$-module $\mathcal{V}(\pi)$, we define its \emph{spectral character} $\chi_{\pi} \in \Xi$ by
\[\chi_{\pi}(\m) = \pi(\m)\, \mod\, Q\]
for $\m \in \Specm A$.
A finite-dimensional $A \otimes \g$-module $V$ is said to have the spectral character $\chi \in \Xi$ if $\chi = \chi_\pi$ for any composition factor $\mathcal{V}(\pi)$ of $V$.\end{dfn}

We denote by $\F$ the category of finite-dimensional $A \otimes \g$-modules.
For each $\chi \in \Xi$ we define the full subcategory $\F_\chi$ of $\F$ whose objects have the spectral character $\chi$.  

\begin{thm}\label{thm:block}
We have the block decomposition $\F = \bigoplus_{\chi \in \Xi} \F_\chi$.
\end{thm}

It suffices to show the following proposition.

\begin{prop}\label{prop:block}
\begin{enumerate}
\item Any finite-dimensional indecomposable $A \otimes \g$-module has some spectral character.

\item Any finite-dimensional simple $A \otimes \g$-modules which have the same spectral character belong to the same block.
\end{enumerate}
\end{prop}

We need two lemmas.

\begin{lem}\label{lem:tensor}
Let $V_1, V_2, V'_1, V'_2$ be finite-dimensional simple $A \otimes \g$-modules and suppose that $V_1$ and $V'_1$ belong to the same block, $V_2$ and $V'_2$ belong to the same block.
Then $V_1 \otimes V_2$ and $V'_1 \otimes V'_2$ belong to the same block. 
\end{lem}
\begin{proof}
We may assume that $V_1=V_1'$.
Put $V=V_2, V'=V_2', M=V_1=V_1'$ for the simplicity of notation.
It suffices to show the following: if $\Ext^1(V, V') \neq 0$ then $\Ext^1(M \otimes V, M \otimes V') \neq 0$.
As explained in Remark~\ref{rem:main}, the natural morphism
\[\Ext^1(V, V') \to \Ext^1(M \otimes V, M \otimes V')\]
coincides with 
\[\Ext^1(V,V') \to \Ext^1(M^* \otimes M \otimes V, V') \simeq \Ext^1(M \otimes V, M \otimes V').\]
Therefore it suffices to show that
\[\Ext^1(V,V') \to \Ext^1(M^* \otimes M \otimes V, V')\]
is injective.
This follows from the fact that the exact sequence
\[\xymatrix{0 \ar[r] & \Ker \ar[r] & M^* \otimes M \ar[r] & \mathcal{V}(0) \ar[r] & 0}\]
splits.
\end{proof}

The following lemma is proved in \cite[Proposition~1.2]{MR2078944}.

\begin{lem}\label{lem:link}
Let $\lambda, \mu \in P^+$ with $\lambda - \mu \in Q$.
Then there exists a sequence $\lambda=\lambda_0, \lambda_1, \ldots, \lambda_r=\mu $ in $P^+$ such that
\[\Hom_\g(\g \otimes V(\lambda_i), V(\lambda_{i+1})) \neq 0\]
for any $i$. 
\end{lem}

\renewcommand{\proofname}{{\rm\textbf{Proof of Proposition \ref{prop:block}.}}}
\begin{proof}
The assertion of (i) immediately follows from Theorem \ref{thm:main}.

We prove (ii).
It suffices to show the assertion for the simple modules of the form $V_\m(\lambda)$ by Lemma~\ref{lem:tensor}.
By Proposition \ref{prop:single} and Lemma \ref{lem:link}, we reduce to claim that $\Der(A, A/\m) \neq 0$.
This is deduced from the following well known facts:
\[\Der(A, A/\m) \simeq \Hom_\c(\m/\m^2, \c)\]
and $\m/ \m^2 = 0$ if and only if $A = \c$.
\end{proof}

\end{document}